\documentclass[fleqn,12pt]{olplainarticle}
\usepackage[all]{xy}
\usepackage{tikz-cd}
\usepackage[mathcal]{euscript}
\newtheorem{thm}{Theorem}[section]

\newtheorem{theorem}{Theorem}

\theoremstyle{definition}
\newtheorem{definition}[thm]{Definition}
\newtheorem{example}[thm]{Example}

\newtheorem{conj}[thm]{Conjecture}

\newtheorem{prop}[thm]{Proposition}
\newtheorem{cor}[thm]{Corollary}
\newtheorem{question}[thm]{Question}

\theoremstyle{remark}
\newtheorem{remark}[thm]{Remark}
\title{Noncommutative Quillen-Lichtenbaum Conjecture}

\author[$\dagger$]{Chunhui Wei}
\affil[$\dagger$]{The University of Melbourne\newline
chunhuiw2@student.unimelb.edu.au\newline chunhuiwei@mail.ustc.edu.cn}

\begin{abstract}
We establish isomorphism ranges for the comparison maps between algebraic and topological K-groups, extending classical Quillen-Lichtenbaum conjecture to separated complex schemes of finite type after refinement. Additionally, we generalizes the conjecture through the lens of noncommutative geometry.
\end{abstract}

\begin{document}

\flushbottom
\maketitle
\thispagestyle{empty}
\tableofcontents
\section*{Introduction}
The Quillen-Lichtenbaum conjecture \cite[Theorem 4.1]{Pedrini_Weibel_2001} compares the algebraic K-theory and topological K-theory of smooth separated complex schemes of finite type. 
\begin{thm}[Quillen-Lichtenbaum conjecture]\quad\\
    Let $X$ is a smooth separated scheme over $\mathbb{C}$ of finite type and $m\geq 1$ be an integer. Then the canonical map
    \[
    K_n(X,\mathbb{Z}/m\mathbb{Z})\to KU^{-n}(X,\mathbb{Z}/m\mathbb{Z})
    \]
    is an isomorphism for $n\geq\dim X-1$ and a monomorphism for $n=\dim X-2$.
\end{thm}
The Quillen-Lichtenbaum Conjecture is now a consequence of the Bloch-Kato conjecture (Readers can see \cite{Friedlander_Haesemeyer_Walker_2004} for details), which has been proved by Voevodsky \cite{Voevodsky_2008}. Hence, it is already a theorem.

Our main result is:
\begin{theorem}\label{mainresult}
    Let $X$ is a separated scheme over $\mathbb{C}$ of finite type and $m\geq 1$ be an integer. Then the canonical map
    \[
    K_n(\mathrm{coh}(X),\mathbb{Z}/m\mathbb{Z})\to K^{\mathrm{top}}_n(\mathrm{coh}(X),\mathbb{Z}/m\mathbb{Z})
    \]
    is an isomorphism for $n\geq\dim X-1$ and a monomorphism for $n=\dim X-2$.
\end{theorem}
\begin{remark}
    Here, $K(\mathrm{coh}(X))$ is Quillen's algebraic K-theory \cite{Quillen_1973} of abelian category $\mathrm{coh}(X)$ or Blumberg-Gepner-Tabuada's algebraic K-theory \cite{blumberg_gepner_tabuada_2013} of stable $\infty$-category $D^b(\mathrm{coh}(X))$. It is so-called the $K'$-theory or $G$-theory of scheme $X$. $K^\mathrm{top}(\mathrm{coh}(X))$ is Blanc's topological K-theory \cite{blanc_2015} of complex dg category $D^b(\mathrm{coh}(X))$.
\end{remark}
\begin{remark}
    By \cite{blanc_2015}, when $X$ is smooth and of finite type, this theorem is the same as Quillen-Lichtenbaum conjecture.
\end{remark}
\begin{conj}[Semi-topological Quillen-Lichtenbaum conjecture]\quad\\
    Let $X$ is a smooth separated scheme over $\mathbb{C}$ of finite type and $A$ be an abelian group. Then the canonical map
    \[
    K^\mathrm{sst}_n(X,A)\to KU^{-n}(X,A)
    \]
    is an isomorphism for $n\geq\dim X-1$ and a monomorphism for $n=\dim X-2$.
\end{conj}
This motives us to give a conjecture:
\begin{conj}
Let $X$ is a separated scheme over $\mathbb{C}$ of finite type and $A$ be an abelian group. Then the canonical map
    \[
    K^\mathrm{st}_n(\mathrm{coh}(X),A)\to K^{\mathrm{top}}_n(\mathrm{coh}(X),A)
    \]
    is an isomorphism for $n\geq\dim X-1$ and a monomorphism for $n=\dim X-2$.
\end{conj}
\begin{remark}
    $K^\mathrm{st}(\mathrm{coh}(X))$ is Blanc's semi-topological K-theory \cite{blanc_2015} of complex dg category $D^b(\mathrm{coh}(X))$.
\end{remark}
Furthermore, we can consider general $\mathbb{C}$-linear stable $\infty$-categories.
\begin{definition}
    Let $\mathcal{A}$ be a $\mathbb{C}$-linear stable $\infty$-category. Define the \emph{Quillen-Lichtenbaum dimension} by the minimal integer $d\geq 0$ such that for all $m\geq 1$, $K_n(\mathcal{A},\mathbb{Z}/m\mathbb{Z})\to K^\mathrm{top}_n(\mathcal{A},\mathbb{Z}/m\mathbb{Z})$ is an isomorphism for $n\geq d-1$ and a monomorphism for $n=d-2$. Denote it by $\dim_\mathrm{QL}\mathcal{A}$.
\end{definition}
\begin{conj}
    Let $\mathcal{A}$ be a smooth $\mathbb{C}$-linear stable $\infty$-category. Then $\dim_\mathrm{QL}\mathcal{A}<\infty$.
\end{conj}
\section*{Acknowledgments}
I would like to express my sincere gratitude to everyone who supported me throughout the process of completing this article. First and foremost, I am deeply thankful to my Master's supervisor, Professor Christian Haesemeyer, for his invaluable guidance, encouragement, and insightful feedback. Furthermore, I'm very grateful for the support from colleagues at  Institute for Advanced Study in Mathematics, Zhejiang University during my visiting.
\section{$k$-linear stable $\infty$-categories}
Denote by $\mathrm{Cat}^\mathrm{perf}$, the $\infty$-category of small, stable, and idempotent complete $\infty$-categories and exact functors. From \cite[Section 3.1]{blumberg_gepner_tabuada_2013}, $\mathrm{Cat}^\mathrm{perf}$ admits (closed) symmetric monoidal structure.   
\begin{definition}
     Let $k$ be a field. Then $\mathrm{Perf}(k)\in\mathrm{CAlg}^\mathrm{rig}(\mathrm{Cat}^\mathrm{perf})$ is a rigid commutative algebra. We set $\mathrm{Cat}^\mathrm{perf}_k:=\operatorname{Mod}_{\mathrm{Perf}(k)}(\mathrm{Cat}^\mathrm{perf})$.
\end{definition}
\begin{remark}
    For the definition of commutative algebra and its module, refer to \cite[Section 4.5]{lurie_2017}. For the definition of rigidness, refer to \cite[section 4]{Hoyois_2017}.
\end{remark}
$\mathrm{Cat}^\mathrm{perf}_k$ also admits (closed) symmetric monoidal structure, denoted by $\otimes_k$.  
\begin{definition}\cite[Definition 5.12, Proposition 5.15]{blumberg_gepner_tabuada_2013}\quad\\
     An \textit{exact sequence} of small stable idempotent-complete $\infty$-categories $\mathcal{A}\to\mathcal{B}\to\mathcal{C}$ is a sequence where the composite morphism is zero, the functor $\mathcal{A} \to \mathcal{B}$ is fully faithful, and the morphism from $\mathcal{B}/\mathcal{A}$ to $\mathcal{C}$ constitutes an equivalence after idempotent-completion. The map $\mathcal{B}\to\mathcal{C}$ is called a \emph{Verdier projection}.
\end{definition}
\begin{definition}\cite[Definition 5.18]{blumberg_gepner_tabuada_2013}\quad\\
    A \textit{split-exact} sequence of small stable $\infty$-categories is an exact sequence $\mathcal{A}\xrightarrow{f}\mathcal{B}\xrightarrow{g}\mathcal{C}$ with exact functors $i:\mathcal{B}\to\mathcal{A}$ and $j:\mathcal{C}\to\mathcal{B}$, which serve as right adjoints to $f$ and $g$, respectively, such that $i\circ f\simeq\mathrm{Id}$ and $g\circ j\simeq\mathrm{Id}$.
\end{definition}
\begin{definition}\cite[Definition 5.3]{Hoyois_2017}\quad\\
A sequence
\[
\mathcal{A} \xrightarrow{f} \mathcal{B} \xrightarrow{g} \mathcal{C}
\]
in $\mathrm{Cat}^\mathrm{perf}_k$ is called \emph{exact} (resp. \emph{split exact}) if its image by the forgetful functor $\mathrm{Cat}^\mathrm{perf}_k \to \mathrm{Cat}^\mathrm{perf}$ is exact (resp. split exact).
\end{definition}
\begin{definition}\cite[Definition 5.11]{Hoyois_2017}
    Let $\mathcal{X}$ be a stable presentable $\infty$-category and let $\Theta: \mathrm{Cat}^\mathrm{perf}_k \to \mathcal{X}$ be a functor. We say that $\Theta$ is an \emph{additive invariant} if the following conditions are satisfied:
\begin{enumerate}
\item $\Theta$ preserves filtered colimits.
\item $\Theta$ preserves zero objects.
\item $\Theta$ sends split exact sequences in $\mathrm{Cat}^\mathrm{perf}_k$ to cofiber sequences in $\mathcal{X}$.
\end{enumerate}
\end{definition}
\begin{definition}\cite[Definition 5.16]{Hoyois_2017}
    Let $\mathcal{X}$ be a stable presentable $\infty$-category and let $\Theta: \mathrm{Cat}^\mathrm{perf}_k\to \mathcal{X}$ be a functor. We say that $\Theta$ is an \emph{localizing invariant} if the following conditions are satisfied:
\begin{enumerate}
\item $\Theta$ preserves filtered colimits.
\item $\Theta$ preserves zero objects.
\item $\Theta$ sends exact sequences in $\mathrm{Cat}^\mathrm{perf}_k$ to cofiber sequences in $\mathcal{D}$.
\end{enumerate}
\end{definition}
\begin{definition}
    Let $\Theta: \mathrm{Cat}^\mathrm{perf}_k \to \mathcal{X}$ be a functor.
    \begin{itemize}
        \item Let $R$ be a $k$-algebra, denote by $\Theta(A):=\Theta(\mathrm{Perf}(R))$.
        \item Let $X$ be a $k$-scheme, denote by $\Theta(A):=\Theta(\mathrm{Perf}(X))$.
        \item Let $A$ be a $k$-linear abelian category, denote by $\Theta(A):=\Theta(D^b(A))$.
    \end{itemize}
\end{definition}
\begin{example}
    The algebraic $K$-theory $K:\mathrm{Cat}^\mathrm{perf}_k\to\mathrm{Sp}$ and Weibo's homotopy $K$-theory $KH:\mathrm{Cat}^\mathrm{perf}_k\to\mathrm{Sp}$ are two localizing invariants.
\end{example}
\section{Topological K-theory of $\mathbb{C}$-linear stable $\infty$-categories}
Denote by $\mathrm{Aff}_\mathbb{C}$ the category of affine $\mathbb{C}$-schemes of finite type. The topological realization $\mathrm{Re}:\mathcal{P}_{\mathrm{Sp}}(\mathrm{Aff}_\mathbb{C})\to\mathrm{Sp}$ is defined as the left Kan extension \cite{Antieau_Heller_2017}
\[
\begin{tikzcd}
\mathrm{Aff}_\mathbb{C}
\arrow[rr, "X \mapsto \Sigma^\infty \mathrm{Sing} X^{\mathrm{an}}_+"]
\arrow[d]
&& \mathrm{Sp} \\
\mathcal{P}_{\mathrm{Sp}}(\mathrm{Aff}_\mathbb{C})
\arrow[rru, dashed, "{\mathrm{Re}}"]
&
\end{tikzcd},
\]
where the left vertical functor $\mathrm{Aff}_\mathbb{C}\to\mathcal{P}_{\mathrm{Sp}}(\mathrm{Aff}_\mathbb{C})$ is the spectral Yoneda functor and $\mathcal{P}_{\mathrm{Sp}}(\mathrm{Aff}_\mathbb{C})$ is the stable presentable $\infty$-category of presheaves of spectra on $\mathrm{Aff}_\mathbb{C}$. $\mathrm{Re}$ can be described as the following composition \cite{Antieau_Heller_2017}
\[
\mathcal{P}_{\mathrm{Sp}}(\mathrm{Aff}_\mathbb{C})\to\mathcal{P}_{\mathrm{Sp}}(\mathrm{Top})\to\mathcal{P}_{\mathrm{Sp}}(\Delta)\to\mathrm{Sp}.
\]
It means, for a presheaf $F:\mathrm{Aff}_\mathbb{C}\to\mathrm{Sp}$, we have
\[
\mathrm{Re}(F)\simeq|[n] \mapsto \mathop{\mathrm{colim}}\limits_{\Delta^n \to (\mathrm{Spec} R)^{\text{an}}\in(\cdot)^{\text{an}} / \Delta_{\mathrm{top}}^n} F(R)
|.
\]
Denote $\mathrm{Sm}_\mathbb{C}$ the subcategory of $\mathrm{Aff}_\mathbb{C}$ containing smooth affine $\mathbb{C}$-schemes of finite type. Then, we can also define a realization functor $\mathrm{Re}^{\mathrm{sm}}:\mathcal{P}_{\mathrm{Sp}}(\mathrm{Sm}_\mathbb{C})\to\mathrm{Sp}$. For a presheaf $F:\mathrm{Aff}_\mathbb{C}\to\mathrm{Sp}$, we have $\mathrm{Re}(F)\simeq\mathrm{Re}^{\mathrm{sm}}(F|_{\mathrm{Sm}_\mathbb{C}})$ \cite[Theorem 3.18]{blanc_2015}.

\begin{definition}
For a functor $F:\mathrm{Cat}^\mathrm{perf}_\mathbb{C}\to\mathrm{Sp}$, define a functor $\underline{F}:\mathrm{Cat}^\mathrm{perf}_\mathbb{C}\to\mathcal{P}_\mathrm{Sp}(\mathrm{Aff}_\mathbb{C})$
\[
\underline{F}(\mathcal{A})(\mathrm{Spec} R):=F(\mathcal{A}\otimes_{\mathbb{C}}\mathrm{Perf}(R)).
\]
And, define a functor $F^\mathrm{st}:\mathrm{Cat}^\mathrm{perf}_\mathbb{C}\to\mathrm{Sp}$ by
\[
F^\mathrm{st}(\mathcal{A}):=\mathrm{Re}(\underline{F}(\mathcal{A})).
\]
\end{definition}
For a $\mathbb{C}$-linear stable $\infty$-category $\mathcal{A}$, define presheaves of spectra as follows
    \[
        \underline{K}(\mathcal{A}):\mathrm{Aff}^{\mathrm{op}}_{\mathbb{C}}\to\mathrm{Sp},\mathrm{Spec} R\mapsto K(\mathcal{A}\otimes_{\mathbb{C}}\mathrm{Perf}(R)),
        \]
    
Then the \textit{semi-topological K-theory} of a $\mathbb{C}$-linear stable $\infty$-category $\mathcal{A}\in\mathrm{Cat}^\mathrm{perf}_{\mathbb{C}}$ is defined as follows
    \[
    K^{\mathrm{st}}(\mathcal{A}):=\mathrm{Re}(\underline{K}(\mathcal{A})).
    \]
These define a localizing invariant:
    \[
    K^{\mathrm{st}}:\mathrm{Cat}^\mathrm{perf}_{\mathbb{C}}\to\mathrm{Sp}.
    \]
Hence,
    \[
    K^\mathrm{st}(\mathcal{A}):=|[n]\mapsto\mathop{\mathrm{colim}}\limits_{\Delta^n \to (\mathrm{Spec} R)^{\text{an}}\in(\cdot)^{\text{an}} / \Delta_{\mathrm{top}}^n} K(\mathcal{A}\otimes_{\mathbb{C}}\mathrm{Perf}(R))|.
    \]
By \cite[Lemma 3.1]{Antieau_Heller_2017}, we also know
    \[
    K^\mathrm{st}(\mathcal{A}):=|[n]\mapsto\mathop{\mathrm{colim}}\limits_{\Delta^n \to (\mathrm{Spec} R)^{\text{an}}\in(\cdot)^{\text{an}} / \Delta_{\mathrm{top}}^n} KH(\mathcal{A}\otimes_{\mathbb{C}}\mathrm{Perf}(R))|.
    \]
\begin{thm}\cite{Antieau_Heller_2017}
    The functor $K^{\mathrm{st}}:\mathrm{Cat}^\mathrm{perf}_{\mathbb{C}}\to\mathrm{Sp}$ satisfies the following properties.
    \begin{enumerate}
        \item For a smooth separated $\mathbb{C}$-scheme $X$ of finite type, there exists a functorial isomorphism
    \[
    K^\mathrm{st}(\mathrm{Perf}(X))\simeq K^\mathrm{sst}(X).
    \]
        \item $K^{\mathrm{st}}/n\simeq K/n\simeq KH/n$ for $n\in\mathbb{Z}_{\geq 1}$.
    \end{enumerate}
\end{thm}
Let $\beta\in\pi_2\mathrm{bu}$ be the Bott element. Then the \textit{topological K-theory} of a $\mathbb{C}$-linear stable $\infty$-category $\mathcal{A}\in\mathrm{Cat}^\mathrm{perf}_{\mathbb{C}}$ is defined as:
    \[K^{\mathrm{top}}(\mathcal{A}):=K^{\mathrm{st}}(\mathcal{A})[\beta^{-1}]\]
This defines a functor
    \[K^{\mathrm{top}}:\mathrm{Cat}^\mathrm{perf}_{\mathbb{C}}\to\mathrm{Sp}.\]
\begin{thm}\cite{blanc_2015}
    The functor $K^{\mathrm{top}}:\mathrm{Cat}^\mathrm{perf}_{\mathbb{C}}\to\mathrm{Sp}$ satisfies the following properties.
    \begin{enumerate}
        \item $K^{\mathrm{top}}(\mathrm{Perf}(\mathbb{C}))\simeq\mathrm{KU}$ in the homotopy category of $\mathrm{Sp}$.
        \item For a separated $\mathbb{C}$-scheme $X$ of finite type, there exists a functorial isomorphism
        \[
        K^{\mathrm{top}}(\mathrm{Perf}(X))\simeq KU(X^{\mathrm{an}}).
        \]
        \item $K^{\mathrm{top}}$ is a localizing invariant.
    \end{enumerate}
\end{thm}
\section{BGQ spectral sequence}
t-structure is firstly introduced by \cite[Definition 1.3.1]{Deligne_1982}. \cite[Definition 1.2.1.1]{lurie_2017} gives the definition of t-structure on stable $\infty$-categories with homological indexing.
\begin{definition}
A \emph{$t$-structure} on a stable $\infty$-category $\mathcal{C}$ consists of a pair of full subcategories $\mathcal{C}_{\ge 0} \subseteq E$ and $\mathcal{C}_{\le 0} \subseteq \mathcal{C}$ satisfying the following conditions:
\begin{enumerate}
    \item $\mathcal{C}_{\ge 0}[1] \subseteq \mathcal{C}_{\ge 0}$ and $\mathcal{C}_{\le 0} \subseteq \mathcal{C}_{\le 0}[1]$;
    \item if $x \in \mathcal{C}_{\ge 0}$ and $y \in \mathcal{C}_{\le 0}$, then $\operatorname{Hom}_\mathcal{C}(x, y[-1]) = 0$;
    \item every $x \in \mathcal{C}$ fits into a cofiber sequence $\tau_{\ge 0}x \to x \to \tau_{\le -1}x$ where $\tau_{\ge 0}x \in \mathcal{C}_{\ge 0}$ and $\tau_{\le -1}x \in \mathcal{C}_{\le 0}[-1]$.
\end{enumerate}
\end{definition}
If $\mathcal{D}$ is another stable category with a $t$-structure $(\mathcal{D}_{\ge 0}, \mathcal{D}_{\le 0})$, then an exact functor $F : \mathcal{C} \to \mathcal{D}$ is called \emph{left $t$-exact} if $F(\mathcal{C}_{\le 0}) \subset \mathcal{D}_{\le 0}$, and \emph{right $t$-exact} if $F(\mathcal{C}_{\ge 0}) \subset \mathcal{D}_{\ge 0}$. Further, $F$ is called \emph{$t$-exact} if it is both left and right $t$-exact.

Given a $t$-structure $(\mathcal{C}_{\ge 0}, \mathcal{C}_{\le 0})$ on a stable category $\mathcal{C}$, for any $a \in \mathbb{Z}$ we put
\[
\mathcal{C}_{\ge a} = \mathcal{C}_{\ge 0}[a], \quad \mathcal{C}_{\le a} = \mathcal{C}_{\le 0}[a].
\]
For integers $a \leq b$ we put $\mathcal{C}_{[a,b]} = \mathcal{C}_{\ge a} \cap \mathcal{C}_{\le b}$. In particular, $\mathcal{C}_{[0,0]} = \mathcal{C}^\heartsuit$ is the heart of the $t$-structure, which is an abelian category. We recall the notation
\[
\mathcal{C}^+ = \bigcup_{a \in \mathbb{Z}} \mathcal{C}_{\le a}, \quad \mathcal{C}^- = \bigcup_{a \in \mathbb{Z}} \mathcal{C}_{\ge a}, \quad \mathcal{C}^b = \mathcal{C}^+ \cap \mathcal{C}^- = \bigcup_{n \ge 0} \mathcal{C}_{[-n,n]}.
\]
The $t$-structure $(\mathcal{C}_{\ge 0}, \mathcal{C}_{\le 0})$ is called \emph{left bounded} resp. \emph{right bounded} resp. \emph{bounded} if $\mathcal{C} = \mathcal{C}^+$ resp. $\mathcal{C} = \mathcal{C}^-$ resp. $\mathcal{C} = \mathcal{C}^b$. We use the notation from \cite{efimov2026}:
\begin{definition}
    A small \emph{$t$-category} is a small stable category $\mathcal{C}$ with a bounded $t$-structure $(\mathcal{C}_{\ge 0}, \mathcal{C}_{\le 0})$.
\end{definition}
\begin{definition}
    An $t$-exact functor $F:\mathcal{C}\to\mathcal{D}$ between small $t$-categories is called \emph{coconnective} if
        \begin{itemize}
            \item $F(\mathcal{C})$ generates $\mathcal{D}$ as a stable idempotent-complete $\infty$-categories.
            \item $F|_{\mathcal{C}^\heartsuit}:\mathcal{C}^\heartsuit\to\mathcal{D}^\heartsuit$ is fully faithful.
        \end{itemize}
\end{definition}

\begin{definition}
    An additive invariant $\Theta:\mathrm{Cat}^\mathrm{perf}_k\to\mathrm{Sp}$ is called {$t$-invariant} if for every coconnective $t$-exact functor $F$ between small $t$-categories, $\Theta(F)$ is an equivalence.
\end{definition}
The the following result is just from \cite[Theorem 4,Corollary 1]{Quillen_1973}
\begin{prop}\label{prop:tinvariant}
    Let $\Theta:\mathrm{Cat}^\mathrm{perf}_k\to\mathrm{Sp}$ be a $t$-invariant additive invariant. Let $A$ be a $k$-linear abelian category such that every object has finite length. Then
    \[
    \Theta_i(A)\simeq \coprod_{j\in J}\Theta_i(D_j),
    \]
    where $\{X_j,j\in J\}$ is a set of representatives for the isomorphism classes of simple objects of $A$, and $D_j$ is the sfield $\mathrm{End}(X_j)^\mathrm{op}$.
\end{prop}
\begin{proof}
    Let $B$ be the full subcategory of all semi-simple objects in $A$. Since $\Theta$ is $t$-invariant, $\Theta_i(B)\simeq\Theta_i(A)$. So we reduce to the case where every objects of $A$ is semi-simple. Since $\Theta$ commutes with direct product and filtered colimits, we reduce to the case where $A$ has a single simple object $X$ up to isomorphism. In this case, $A$ is equivalent to the abelian category of f.g. modules of $\mathrm{End}(X)^\mathrm{op}$.
\end{proof}
The  Brown-Gersten-Quillen(BGQ) spectral sequence is just from \cite[Theorem 7.5.4]{Quillen_1973}.
\begin{thm}\label{thm:BGQ}
    Let $\Theta:\mathrm{Cat}^\mathrm{perf}_k\to\mathrm{Sp}$ be a $t$-invariant localizing invariant. Let $X$ is a separated noetherian scheme over $k$. There is a spectra sequence
    \[
    E_1^{p,q}=\coprod_{x\in X^p}\Theta_{-p-q}(\kappa(x))\implies \Theta_{-n}(\mathrm{coh}(X))
    \]
    which is convergent when $X$ has finite dimension. Here, $X^p$ is the set of points of codimension $p$ in $X$ and $\kappa(x)$ is the residue field at $x$.
\end{thm}
\begin{proof}
    Let $\mathrm{coh}_p(X)$ be the Serre subcategory of $\mathrm{coh}(X)$ consisting of those coherent sheaves whose support is of codimension $\geq p$. There is an equivalence
    \[
    \mathrm{coh}_p(X)/\mathrm{coh}_{p+1}(X)\simeq\coprod_{x\in X^p}\bigcup_n\mathrm{mod}(\mathcal{O}_{X,x}/\mathrm{rad}(\mathcal{O}_{X,x})^n)
    \]
    By Proposition \ref{prop:tinvariant}, we have
    \[
    \Theta_i(\mathrm{coh}_p(X)/\mathrm{coh}_{p+1}(X))\simeq\coprod_{x\in X^p}\Theta_i(\kappa(x)).
    \]
    Since $X$ is localizing invariant, we have long exact sequence:
    \[
    \to\Theta_i(\mathrm{coh}_{p+1}(X))\to\Theta_i(\mathrm{coh}_p(X))\to \coprod_{x\in X^p}\Theta_i(\kappa(x))\to \Theta_{i-1}(\mathrm{coh}_{p+1}(X))\to
    \]
    which give rise to the desired spectra sequence in a standard way.
\end{proof}
\begin{cor}
    Let $X$ is a separated noetherian scheme over $\mathbb{C}$ that has finite dimension and $m\ge 1$ be an integer. There are convergent spectra sequences
    \begin{align}
        E_1^{p,q}(\mathrm{st})&=\coprod_{x\in X^p}K^\mathrm{st}_{-p-q}(\kappa(x),\mathbb{Z}/m\mathbb{Z})\implies K^\mathrm{st}_{-n}(\mathrm{coh}(X),\mathbb{Z}/m\mathbb{Z})\\
        E_1^{p,q}(\mathrm{top})&=\coprod_{x\in X^p}K^\mathrm{top}_{-p-q}(\kappa(x),\mathbb{Z}/m\mathbb{Z})\implies K^\mathrm{top}_{-n}(\mathrm{coh}(X),\mathbb{Z}/m\mathbb{Z})
    \end{align}
\end{cor}
\begin{proof}
    By \cite{Antieau_Heller_2017}, $K^\mathrm{st}/m\simeq KH/m$. By \cite{efimov2026}, $KH$ is $t$-invariant. Hence, $K^\mathrm{st}/m$ is $t$-invariant. Since $K^\mathrm{top}=K^\mathrm{st}[\beta^{-1}]$, $K^\mathrm{top}/m$ is also $t$-invariant.
\end{proof}
\section{Quillen-Lichtenbaum conjecture for singular schemes}
The following theorem is just a consequence of Quillen-Lichtenbaum conjecture in smooth case.
\begin{thm}\label{thm:QLforfield}
    Let $F$ be a finitely generated field extension over $\mathbb{C}$ and $m\geq 1$ be an integer. Then the canonical map
    \[
    K^{\mathrm{st}}_i(F, \mathbb{Z}/m\mathbb{Z}) \to K^{\mathrm{top}}_i(F, \mathbb{Z}/m\mathbb{Z})
    \]
    is an isomorphism for $i \ge \operatorname{trdeg}_{\mathbb{C}}(F) - 1$ and a monomorphism for $i = \operatorname{trdeg}_{\mathbb{C}}(F) - 2$, where $\operatorname{trdeg}_\mathbb{C}(F)$ is its transcendence degree over $\mathbb{C}$.
\end{thm}
\begin{proof}
    For a finitely generated field extension $F$ over $\mathbb{C}$, we can view $F$ as the function field of a smooth separated scheme $X$ of finite type with $\dim X=\operatorname{trdeg}_{\mathbb{C}}(F)$. Geometrically, $\operatorname{Spec} F$ is the inverse limit of all non-empty Zariski open subsets $U_\alpha \subset X$. Algebraically, $F$ is the filtered colimit of their rings of functions:
    \[
    F = \operatorname{colim} \mathcal{O}_X(U_\alpha)
    \]
    Hence,
    \[
    \begin{aligned}
        K^\mathrm{st}(F)&\simeq\operatorname{colim}K^\mathrm{st}(U_\alpha)\simeq\operatorname{colim}K^\mathrm{sst}(U_\alpha)\\
        K^\mathrm{top}(F)&\simeq\operatorname{colim}K^\mathrm{top}(U_\alpha)\simeq \operatorname{colim}KU(U_\alpha)
    \end{aligned}
    \]
    Since $U_\alpha$ is smooth affine separated scheme of finite type with $\dim U_\alpha=\operatorname{trdeg}_{\mathbb{C}}(F)$, we get desired results from Quillen-Lichtenbaum conjecture in smooth case.
\end{proof}
\begin{thm}\label{thm:QLforscheme}
Let $X$ is a separated scheme over $\mathbb{C}$ of finite type and $m\geq 1$ be an integer. Then the canonical map
    \[
    K^\mathrm{st}_n(\mathrm{coh}(X),\mathbb{Z}/m\mathbb{Z})\to K^{\mathrm{top}}_n(\mathrm{coh}(X),\mathbb{Z}/m\mathbb{Z})
    \]
    is an isomorphism for $n\geq\dim X-1$ and a monomorphism for $n=\dim X-2$.
\end{thm}
\begin{proof}
Assume $d=\dim X$. For a point $x \in X^p$, the residue field $F = \kappa(x)$ is a finitely generated field extension of $\mathbb{C}$. Its transcendence degree over $\mathbb{C}$ is exactly the dimension of its closure:
\[
\operatorname{trdeg}_{\mathbb{C}}(F) = d - p.
\]
Hence, by Theorem \ref{thm:QLforfield},  
    \[
    K^{\mathrm{st}}_i(F, \mathbb{Z}/m\mathbb{Z}) \to K^{\mathrm{top}}_i(F, \mathbb{Z}/m\mathbb{Z})
    \]
    is an isomorphism for $i \ge \operatorname{trdeg}_{\mathbb{C}}(F) - 1=d-p-1$ and a monomorphism for $i = \operatorname{trdeg}_{\mathbb{C}}(F) - d-p-2$. Hence, the map $f_1^{p,q}: E_1^{p,q}(\mathrm{st}) \to E_1^{p,q}(\mathrm{top})$ is an isomorphism when $-p-q \ge d - p - 1 \implies q \le 1 - d$, and a monomorphism when $-p-q = d - p - 2 \implies q = 2 - d$.

    The remaining proof is just from the Comparison Theorem of spectra sequences.
\end{proof}
\begin{remark}
    Since $K/m\simeq K^\mathrm{st}/m$ by \cite{Antieau_Heller_2017}, this theorem is just the Theorem \ref{mainresult}.
\end{remark}
\begin{conj}\label{conj:sQLforscheme}
Let $X$ is a separated scheme over $\mathbb{C}$ of finite type and $A$ be an abelian group. Then the canonical map
    \[
    K^\mathrm{st}_n(\mathrm{coh}(X),A)\to K^{\mathrm{top}}_n(\mathrm{coh}(X),A)
    \]
    is an isomorphism for $n\geq\dim X-1$ and a monomorphism for $n=\dim X-2$.
\end{conj}
\begin{conj}\label{conj:sQLforfield}
    Let $F$ be a finitely generated field extension over $\mathbb{C}$ and $A$ be an abelian group. Then the canonical map
    \[
    K^{\mathrm{st}}_i(F, A) \to K^{\mathrm{top}}_i(F, A)
    \]
    is an isomorphism for $i \ge \operatorname{trdeg}_{\mathbb{C}}(F) - 1$ and a monomorphism for $i = \operatorname{trdeg}_{\mathbb{C}}(F) - 2$. 
\end{conj}
\begin{conj}\label{conj:iinvariantforst}
    $K^\mathrm{st}$ is $t$-invariant.
\end{conj}
\begin{thm}\label{thm:sQLforscheme}
     Conjecture \ref{conj:sQLforfield} and Conjecture \ref{conj:iinvariantforst} imply Conjecture \ref{conj:sQLforscheme}.
\end{thm}
\begin{proof}
    The proof is the same as Theorem \ref{thm:QLforscheme}.
\end{proof}
\section{Quillen-Lichtenbaum conjecture of $\mathbb{C}$-linear stable $\infty$-categories}
\begin{definition}
    Let $\mathcal{A}$ be a $\mathbb{C}$-linear stable $\infty$-category. Define the \emph{Quillen-Lichtenbaum dimension} by the minimal integer $d\geq 0$ such that for all $m\geq 1$, $K_n(\mathcal{A},\mathbb{Z}/m\mathbb{Z})\to K^\mathrm{top}_n(\mathcal{A},\mathbb{Z}/m\mathbb{Z})$ is an isomorphism for $n\geq d-1$ and a monomorphism for $n=d-2$. Denote it by $\dim_\mathrm{QL}\mathcal{A}$.
\end{definition}
Theorem \ref{thm:QLforscheme} tells us $\dim_\mathrm{QL}(D^b_\mathrm{coh}(X))\leq\dim X$.
\begin{question}
    For a $\mathbb{C}$-linear stable $\infty$-category, do we have $\dim_\mathrm{QL}\mathcal{A}<\infty$?
\end{question}
We suspect at least for smooth $\mathbb{C}$-linear stable $\infty$-category, its Quillen-Lichtenbaum dimension is finite.
\begin{conj}
    Let $\mathcal{A}$ be a smooth $\mathbb{C}$-linear stable $\infty$-category. Then $\dim_\mathrm{QL}\mathcal{A}<\infty$.
\end{conj}
\begin{definition}
    Let $\mathcal{A}$ be a $\mathbb{C}$-linear stable $\infty$-category. Define the \emph{semi-topological Quillen-Lichtenbaum dimension} by the minimal integer $d\geq 0$ such that $K^\mathrm{st}_n(\mathcal{A})\to K^\mathrm{top}_n(\mathcal{A})$ is an isomorphism for $n\geq d-1$ and a monomorphism for $n=d-2$. Denote it by $\dim_\mathrm{sQL}\mathcal{A}$.
\end{definition}
It is obvious that $\dim_\mathrm{sQL}\mathcal{A}\geq \dim_\mathrm{QL}\mathcal{A}$.
\begin{conj}
    Let $\mathcal{A}$ be a smooth $\mathbb{C}$-linear stable $\infty$-category. Then $\dim_\mathrm{sQL}\mathcal{A}<\infty$.
\end{conj}
To support our conjecture, we find some examples.
\begin{prop}
    If $\mathcal{A}$ is in one of the following cases, then $\dim_\mathrm{QL}(\mathcal{A})<\infty$:
    \begin{enumerate}[label=(\arabic*)]
        \item $\mathcal{A}=\mathrm{Perf}(\mathbb{C}[G])$ for a finite group $G$;
        \item $\mathcal{A}=\mathrm{Perf}(\mathfrak{X})$, where $\mathfrak{X}/\mathbb{C}$ is a derived scheme such that its classical part $\pi_0\mathfrak{X}$ is a smooth separated scheme of finite type.
    \end{enumerate}
\end{prop}
\begin{proof}\quad
\begin{enumerate}[label=(\arabic*)]
    \item The group algebra of a finite group is
    \[
    \mathbb{C}[G] \cong \bigoplus_{i=1}^k M_{d_i}(\mathbb{C})
    \]
    where:$k$ is the number of conjugacy classes of $G$ (which is also the number of distinct irreducible representations), $M_{d_i}(\mathbb{C})$ is the ring of $d_i \times d_i$ matrices with complex entries and $d_i$ is the dimension of the $i$-th irreducible representation. Hence, $K(\mathbb{C}[G])\simeq K(\mathbb{C})^k,K^\mathrm{top}(\mathbb{C}[G])\simeq K^\mathrm{top}(\mathbb{C}[G])$. Thus,
    \[
    \dim_\mathrm{QL}\mathrm{Perf}(\mathbb{C}[G])=0<\infty.
    \]
    \item By \cite[Proposition 6.5]{konovalov2021}, $K^\mathrm{st}(\mathrm{Perf}(\mathfrak{X}))\simeq K^\mathrm{st}(\pi_0\mathfrak{X}),K^\mathrm{top}(\mathrm{Perf}(\mathfrak{X}))\simeq K^\mathrm{top}(\pi_0\mathfrak{X})$. Hence,
    \[
    \dim_\mathrm{QL}\mathrm{Perf}(\mathfrak{X})=\dim_\mathrm{QL}\mathrm{Perf}(\pi_0\mathfrak{X})\leq\dim\pi_0\mathfrak{X}<\infty.
    \]
\end{enumerate}
\end{proof}
\bibliography{refbase}

\end{document}